\documentclass{article}
\usepackage{amssymb,amsmath,amsthm,graphicx}
\usepackage[all,color]{xy}
\usepackage{float}

\def\qed{\hfill {\hbox{${\vcenter{\vbox{               
   \hrule height 0.4pt\hbox{\vrule width 0.4pt height 6pt
   \kern5pt\vrule width 0.4pt}\hrule height 0.4pt}}}$}}}

\newtheorem{theorem}{Theorem}
\newtheorem{lemma}[theorem]{Lemma}
\newtheorem{proposition}[theorem]{Proposition}

\theoremstyle{definition}
\newtheorem{example}{Example}
\newtheorem{definition}{Definition}
\newtheorem{remark}{Remark}

\newcommand{\cQ}{\mathcal{Q}}
\newcommand{\Z}{\mathbb{Z}}
\DeclareMathOperator{\Hom}{Hom}

\def\qed{\hfill {\hbox{${\vcenter{\vbox{               
   \hrule height 0.4pt\hbox{\vrule width 0.4pt height 6pt
   \kern5pt\vrule width 0.4pt}\hrule height 0.4pt}}}$}}}

\def\tr{\triangleright}

\date{}

\title{\Large \textbf{Quandle Action Quivers}}

\author{Mason Cai\footnote{Email: mcab2020@mymail.pomona.edu} \and
Sam Nelson\footnote{Email: Sam.Nelson@cmc.edu. Partially supported by Simons Foundation collaboration grant 702597}}

\begin{document}
\maketitle

\begin{abstract}
\textit{Quandle Coloring Quivers} are directed graph-valued invariants of classical
and virtual knots and links associated to finite quandles. Quandle action quivers
are subquivers of the full quandle coloring quiver associated to quandle actions
by elements of the coloring quandle. These quivers provide a categorification
of the quandle counting invariant associated to each element of the quandle. We 
obtain new polynomial invariants called quandle action polynomials from these quivers
as decategorifications.
\end{abstract}

\parbox{5.5in} {\textsc{Keywords:} 

\smallskip

}

\section{\large\textbf{Introduction}}\label{I}

\textit{Quandles} are algebraic structures with axioms motivated by the
Reidemeister moves in knot theory. Since their introduction in the early 1980s,
quandles have been used to define invariants of knots, links and other 
topological structures.

Every classical or virtual knot or link has a \textit{fundamental quandle}, 
sometimes called the \textit{knot quandle}, analogous to the fundamental 
group of a topological space, given by a presentation that can be read from a 
diagram of the knot or link. In \cite{J}, Joyce showed that the knot quandle 
is complete invariant of classical knots in the sense that its isomorphism 
class determines the knot complement up to homeomorphism. For virtual knots 
and classical and virtual links, the fundamental quandle is a not complete 
but still very strong invariant, determining other invariants such as the 
Alexander polynomial, the knot group and more.

Since comparing isomorphism classes of objects defined by presentations 
is often difficult, in practice we consider the set of quandle homomorphisms 
from the fundamental quandle to a finite quandle $X$, known as the 
\textit{quandle homset}. In \cite{CN} the homset was categorified to
obtain a \textit{quandle coloring quiver} associated to each subset of 
the set of quandle endomorphisms, yielding quiver-valued invariants of
classical and virtual knots and links. Several new polynomial knot invariants
were defined from these quivers via decategorification. Quiver invariants
related to the quandle coloring quiver have been studied in many recent works;
see \cite{CCN, CN, CN2, SCN, KNS, FN, IN} for more.

In this paper we consider the action of a finite quandle on itself.
Since these actions are endomorphisms, they determine a distinguished set of
quandle coloring quivers. Moreover, iterating quandle actions in a finite 
quandle eventually yields an identify map, which leads to distinguished cycles 
in the quiver.
We exploit this to define \textit{quandle action quivers} and for each
element of a finite quandle obtain a new polynomial knot invariant
we call a \textit{quandle action polynomial}.

The paper is organized as follows. In Section \ref{M} we review quandles and
quandle coloring quivers. In Section \ref{QA} we define the quandle action 
quivers and their associated polynomial knot invariants. We provide
examples to illustrate the computation of the new invariants and collect 
some sample results for sets of classical and virtual knots and links.
We conclude with some questions for future research in Section \ref{Q}.

\section{\large\textbf{Quandles and Quandle Actions}}\label{M}

We begin with brief review of quandles; see \cite{EN} for more.

\begin{definition}[quandles]\label{def: quandles}
A \textit{quandle} is a set $X$ equipped with two binary operations 
$\triangleright, \triangleright^{-1}: X \times X \to X$ satisfying the 
following axioms:
\begin{enumerate}
\item[(i)] \label{ax: q1} For all $ x\in X$, $x\triangleright x=x$ (idempotent)
\item[(ii)] \label{ax: q2} For all $x,y\in X$, $(x\triangleright y) \triangleright^{-1}y=x =(x\triangleright^{-1}y)\triangleright y$  (right-invertible)
\item[(iii)] \label{ax: q3} For all $x,y,z\in X$, $(x\triangleright y)\triangleright z= (x\triangleright z)\triangleright(y\triangleright z)$ (self-distributive)
\end{enumerate}
\end{definition}
These axioms arise from the three Reidemeister moves when viewing the elements 
of $X$ as labels or colors for each arc of a knot or link diagram. Then, the 
$\triangleright $ and $ \triangleright^{-1}$ operations can be interpreted as 
the over-arc acting on the under-arc at a crossing from the right or from 
the left.

\begin{example}
Let $G$ be a group. Then the underlying set $G$ equipped with 
$\triangleright$ defined as conjugation by $y$, $x\triangle y = y^{-1} x y$ and 
with $\tr^{-1}$ defined as conjugation by $y^{-1}$,
$x\triangleright^{-1} y = y x y^{-1}$ is a quandle known as the \textit{conjugation 
quandle} of $G$. 
\end{example}

\begin{example}
Let $\Lambda=\Z[t,t^{-1}]$ and consider a $\Lambda$-module $M$. Then $M$ 
equipped with the operation $x\triangleright y= tx+(1-t)y$ with $x,y\in M$ is 
called an \textit{Alexander quandle.} 
\end{example}

Finite quandles can be represented by their operation tables. Consider the 
following example:

\begin{example}\label{ex: dihedral}
The \textit{dihedral quandle} of order $n$ is a finite quandle defined as 
the set $X=\Z_n$ with the operation $x\triangleright y= 2y-x (\mod n)$. 

    The operation table for the dihedral quandle of order $3$ is shown below.
    \[
    \begin{tabular}{c| c c c}
         $\triangleright$ & 0 & 1 & 2 \\
         \hline
         0 & 0 & 2 & 1 \\
         1 & 2 & 1 & 0 \\
         2 & 1 & 0 & 2
    \end{tabular}
    \]
\end{example}

\begin{definition}
Let $L$ be an oriented classical or virtual knot or link. At each crossing of $L$, 
assign the following relation between arc labels:
\[\includegraphics{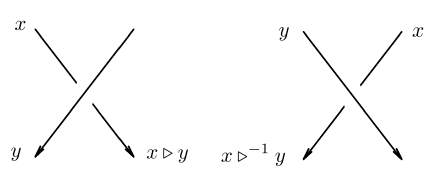}\]

In other words, we can think of the relation as $x$ goes under $y$ to yield 
$x\tr y$. The \textit{fundamental quandle} of $L$, denoted $\cQ(L)$, is the 
quandle generated by the set of arc labels under the equivalence relation
defined by the quandle axioms together with the crossing relations for each 
crossing of $L.$
\end{definition}

\begin{definition}
Let $L$ be an oriented classical or virtual knot or link. Let $X$ be a 
finite quandle. The \textit{counting invariant}, denoted  
\[\Phi_X^{\mathbb{Z}}(L)=|\Hom (\cQ(L),X)|,\] is the 
number of quandle homomorphisms from the fundamental quandle $\cQ(L)$ to 
the coloring quandle $X$. 
\end{definition}

We then have the following standard result (see \cite{EN,J}):

\begin{theorem}
The quandle counting invariant is an invariant of oriented classical and 
virtual knots and links.
\end{theorem}

\begin{remark}
A choice of diagram $D$ for an oriented knot or link $L$ allows us to represent
the homset $\mathrm{Hom}(\cQ(L),X)$ in a concrete way as the set of 
$X$-colorings
of $D$, analogously to how a choice of basis for a vector space enables us to 
represent linear transformations concretely as matrices. A different choice of
diagram yields a different set of representatives for each coloring just as a 
different choice of basis yield different matrices represent the same linear 
transformation. The role of change of basis matrices is played by $X$-colored
Reidemeister moves in this analogy.
\end{remark}

\begin{definition}\label{def: action}
Let $L$ be an oriented knot or link, $D$ a choice of diagram of $L$ with arcs 
numbered $1,\dots, n$ and $X$ a finite quandle and $x\in X$. We can represent 
a homset element
as a vector $f=(x_1,\dots,x_n)\in X^n$ with reference to $D$. Then there is a  
natural \textit{quandle action} of $x$ on the homset $\Hom(\cQ(L),X)$  
defined on elements $f=(x_1,x_2,\dots,x_n)\in \Hom(\cQ(L),X)$ by 
\[x\cdot f=( x_1 \triangleright x, x_2 \triangleright x,\dots, 
x_n \triangleright x).\]
\end{definition}

\begin{lemma}
Let $L,X$ be the same as in Definition \ref{def: action}. Then the natural 
quandle action on $f$ yields a valid coloring for every $f\in \Hom(\cQ(L),X)$.  
\end{lemma}

\begin{proof}
        
    The action of $x$ on the $f$ preserves crossing relations of $L$. This 
follows from the self-distributivity axiom 
\[(x\triangleright z) \triangleright(y\triangleright z)
=(x\triangleright y)\triangleright z.\]
\end{proof}

Quandle actions on the homset induce endomorphisms of 
$X$. Thus, the natural quandle action of any element $x\in X$ yields a quandle 
coloring quiver, $Q_X^x(L)$ as introduced in \cite{CN}. We review 
this definition in the next section.

\section{\large\textbf{Quiver Action Enhancement}}\label{QA}

\textit{Quandle coloring quivers}, $\cQ_X^S(L)$ were defined in 
\cite{CN} in the following way: given a finite quandle $X$ and 
set $S\subset \mathrm{Hom}(X, X)$ of quandle endomorphisms, we can make 
a directed graph 
with a vertex for each $X$-coloring of $L$ and a directed edge from $v_j$ to 
$v_k$ whenever $v_k=f(v_j)$ in the sense that each arc color in $v_k$ is 
obtained from the corresponding arc color in $v_j$ by applying $f$ for
some $f \in S$. For this paper, we specialize the quandle coloring quiver 
to the notion of quandle actions to define the \textit{quandle action quiver.}

\begin{definition}
Let $X$ be a finite quandle and $L$ an oriented classical or virtual knot 
or link represented by a diagram $D$. The \textit{quandle action quiver} 
$\mathcal{Q}_X^A(L)$ has a vertex for each element 
$v\in\mathrm{Hom}(\mathcal{Q}(L),X)$ 
represented by an $X$-coloring of $D$ and for each $x\in X$ an edge 
weighted by $x$ from $v$ to $x\cdot v$, the vertex associated to the
$X$-coloring of $D$ obtained from $v$ by replacing each color $y$ with
$y\tr x$.
\end{definition}

\[\includegraphics{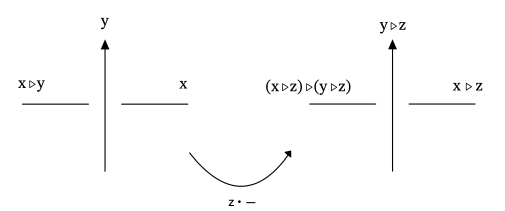}\]

\begin{example}\label{ex:triv}
Let $X$ be the trivial quandle structure on $\{1,2,\dots, n\}$, i.e. $x\tr y=x$
for all $x,y\in X$. Then the homset consists of $n$ monochromatic colorings and 
each element of $X$ acts trivially on these colorings, so the quiver is a
is a disjoint union of $n$ loops.
\end{example}

Let us abbreviate the $n$-times iterated action of $x$ on $v$ as $x \cdot^n v.$

\begin{definition}
For each $x\in X$, the polynomial $\Phi_{X}^{A,x}(L)$ 
\[\Phi_{X}^{A,x}(L)=\sum_{v\in \mathrm{Hom}(\mathcal{Q}(L),X)} u^{l(v,x)}\]
where $l(v,x)$ is the smallest integer $n>0$ such that $x\cdot^n v=v$,
i.e., the length of the loop based at $v$ consisting of edges labeled
with $x$, is the \textit{quandle action polynomial} of $L$ associated
to the element $x\in X$.
\end{definition}

\begin{theorem}
Let $X$ be the trivial quandle on $n$ elements as defined in Example 
\ref{ex:triv}. Then for every $x\in X$ we have $\Phi_X^{A,x}(L)=nu$.
\end{theorem}

\begin{proof}
The follows from Example \ref{ex:triv}.
\end{proof}

More generally, we have:

\begin{theorem}
Let $X$ be a quandle with an element $x\in X$ which acts trivially on $X$, i.e.
such that $y\tr x=y$ for all $y\in X$. Then $\Phi_X^{A,x}(L)=\Phi_X^{\mathbb{Z}(L)}u$.
\end{theorem}

The following example demonstrates how to compute the quandle action polynomial 
for $3_1$ where the coloring quandle $X$ is the dihedral quandle of order 3.

\begin{example}
    Let $X$ be the dihedral quandle of order 3 as in Example \ref{ex: dihedral} 
and let $L$ be the trefoil knot $3_1.$ The homset is 
\[\includegraphics[width=0.5\textwidth]{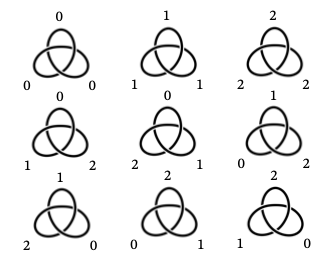}\]
and the quandle action quiver is
\[\includegraphics[width=0.35\textwidth]{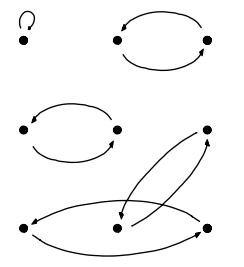}.\]
Then we compute $\Phi_{X}^{A,0}(3_1)=8u^2+u$. 
\end{example}

Since the quandle action polynomial arises from the quandle coloring quiver, it is 
necessarily an invariant of knots and links. 
\begin{theorem}
    Let $X$ be a finite quandle, $L$ an oriented classical or virtual knot or link. 
Let $X$ act naturally on $\Hom(\cQ(L),X)$. Then for each $x\in X$, 
$\Phi_{X}^{A,x}(L)$ is an invariant of $L.$
\end{theorem}

\begin{proof}
$\Hom(\cQ(L),X)$ is an invariant of $L$ and $\Phi_{X}^{A,x}(L)$ is determined up to 
isomorphism of $X$. 
\end{proof}

Examples \ref{ex: prop_enhance_1} and \ref{ex: prop_enhance_2} verify that the 
quandle action polynomial is a proper enhancement of the counting invariant.
\begin{example}\label{ex: prop_enhance_1}
    Let $X$ be the quandle with multiplication table:
\[
    \begin{tabular}{c| c c c c  }
         $\triangleright$ & 1 & 2 & 3 & 4  \\
         \hline
         1 & 1 & 4 & 4 & 1 \\
         2 & 3 & 2 & 2 & 3 \\
         3 & 2 & 3 & 3 & 2 \\
         4 & 4 & 1 & 1 & 4 \\
         
    \end{tabular}
\]
Then we compute the quandle action polynomials for some small links:
\[
\begin{tabular}{c|c}
    $L$ & $\Phi_{X}^{A,4}(L)$ \\
     \hline
     L4a1 & $12u^2+4u$ \\
     L5a1 & $12u^2+4u$ \\
     L6a1 & $12u^2+4u$ \\
     L6a5 & $8u^2+8u$ \\
     L6n1 & $8u^2+8u$ \\
     L7a1 & $12u^2+4u$ \\
     L7a2 & $12u^2+4u$ \\
     L7a3 & $12u^2+4u$ \\
     L7a4 & $12^2u+4u$ \\
     L7a7 & $8u^2+8u$ \\
     L7n1 & $12^2u+4u$ \\
     L7n2 & $12u^2+4u$ \\
     \end{tabular}
\]

The counting invariant yields the same value of $16$ for all the links above, but 
the quandle action polynomial distinguishes many pairs of links. This verifies that 
the quandle action polynomial is indeed a proper enhancement of the counting invariant. 
\end{example}

\begin{example}\label{ex: prop_enhance_2}
Let $X$ be the quandle with multiplication table
\[
    \begin{tabular}{c| c c c c c c }
         $\triangleright$ & 1 & 2 & 3 & 4 & 5 & 6 \\
         \hline
         1 & 1 & 1 & 2 & 2 & 2 & 2 \\
         2 & 2 & 2 & 1 & 1 & 1 & 1 \\
         3 & 3 & 3 & 3 & 6 & 4 & 5 \\
         4 & 4 & 4 & 5 & 4 & 6 & 3\\
         5 & 5 & 5 & 6 & 3 & 5 & 4\\
         6 & 6 & 6 & 4 & 5 & 3 & 6
    \end{tabular}
\]
We compute the quandle action polynomials for some small links:
\[
\begin{tabular}{c|c}
    $L$ & $\Phi_{X}^{A,4}(L)$ \\
     \hline
L4n1 & $12u^6 + 15u^3 + 8u^2 + u$\\
L6n5 & $18u^6 + 3u^3 + 14u^2 + u$\\ 
L7n1 & $12u^6 + 15u^3 + 8u^2 + u$\\ 
L7n4 & $12u^6 + 15u^3 + 8u^2 + u$\\ 
\end{tabular}
\]
and find that while the quandle action polynomial distinguishes L4a1 
from L6a5, L7a1, and L7a4 while the counting invariant for each is $36.$ 
\end{example}

\begin{remark}
The quandle action polynomial is an infinite family of invariants of knots 
and links. For each choice of finite coloring quandle, a fixed acting element 
yields a set of quandle action polynomials valued at each knot. A different 
acting element may yield a different set of quandle action polynomials. 
\end{remark}

More precisely, we have:
\begin{proposition}
Let $X$ be a quandle. If the action of $x\in X$ agrees with the action
of $y\in X$, i.e. if $z\tr x=z\tr y$ for all $z\in X$, then 
\[\Phi_{X}^{X,x}(L) =\Phi_{X}^{X,y}(L).\]
\end{proposition}

\begin{proof}
Immediate from the definition.
\end{proof}

\begin{example}
Let $X$ be the quandle structure on the set $\{1,2,3,4,5\}$ given by the
operation table
\[\begin{array}{r|rrrrr}
\tr & 1 & 2 & 3 & 4 & 5 \\ \hline
1 & 1 & 3 & 1 & 3 & 3 \\
2 & 5 & 2 & 5 & 2 & 2 \\
3 & 3 & 1 & 3 & 1 & 1 \\
4 & 2 & 4 & 2 & 4 & 4 \\
5 & 4 & 5 & 4 & 5 & 5
\end{array}.
\]
The actions of 1 and 3 are the same since their columns in the operation 
matrix agree; similarly, the actions of 2, 4 and 5 are the same.
We computed the quandle action polynomials for elements $1,2\in X$
for the prime classical links with up to seven crossings; the results are 
in the table.
\[\begin{array}{c|cc}
L & \Phi_{X}^{X,1}(L) & \Phi_{X}^{X,2}(L)  \\ \hline
L2a1 & 4u+9u^3 & 9u+4u^2 \\
L4a1 & 4u+9u^3 & 9u+4u^2 \\
L5a1 & 4u+21u^3 & 9u+16u^2 \\
L6a1 & 4u+9u^3 & 9u+4u^2 \\
L6a2 & 4u+9u^3 & 9u+4u^2 \\
L6a3 & 4u+9u^3 & 9u+4u^2 \\
L6a4 & 8u+117u^3 & 27u+98u^2 \\ 
L6a5 & 8u+27u^3 & 27u+8u^2 \\
L6n1 & 8u+27u^3 & 27u+8u^2 \\
L7a1 & 4u+21u^3 & 9u+16u^2 \\
L7a2 & 4u+9u^3 & 9u+4u^2 \\
L7a3 & 4u+21u^3 & 9u+16u^2 \\
L7a4 & 4u+21u^3 & 9u+16u^2 \\
L7a5 & 4u+9u^3 & 9u+4u^2 \\
L7a6 & 4u+9u^3 & 9u+4u^2 \\
L7a7 & 8u+27u^3 & 27u+8u^2 \\
L7n1 & 4u+9u^3 & 9u+4u^2 \\
L7n2 & 4u+21u^3 & 9u+16u^2 
\end{array}.\]
\end{example}

We conclude with an observation.

\begin{remark}
Most decategorifications from quandle coloring quivers lose some
information so that the isomorphism class of the quiver is a stronger
invariant than the polynomial. In the case of quandle action quivers
and polynomials, however, no information in lost in going from the 
quiver to the polynomial; the polynomial determines the quiver up to
graph isomorphism.
\end{remark}

\section{\large\textbf{Questions}}\label{Q}

We end with some questions for future research. 

\begin{itemize}
\item What is the most effective way to combine the quandle action polynomials 
for different elements in a quandle into a single invariant?
\item What other enhancements can be defined from the quandle action quiver?
\item Characterize what kinds of polynomials can be the quandle action 
polynomial of a knot or link with respect to a given quandle. For example,
the coefficient of $u^j$ must be a multiple of $j$ since each of the $j$ 
vertices in a cycle of length $j$ contributes $u^j$ to the polynomial; what 
other properties must the polynomials have?
\item Fully characterize the conditions under which two quandle elements
determine the same quandle action quiver/polynomial.
\end{itemize}


\bibliography{mcai-sn}{}

\begin{thebibliography}{1}

\bibitem{CCN}
J.~Ceniceros, A.~Christiana, and S.~Nelson.
\newblock Psyquandle coloring quivers.
\newblock {\em J. Knot Theory Ramifications}, 32(11):Paper No. 2350073, 18,
  2023.

\bibitem{CN2}
K.~Cho and S.~Nelson.
\newblock Quandle cocycle quivers.
\newblock {\em Topology Appl.}, 268:106908, 10, 2019.

\bibitem{CN}
K.~Cho and S.~Nelson.
\newblock Quandle coloring quivers.
\newblock {\em Journal of Knot Theory and Its Ramifications}, 28(01):1950001,
  2019.

\bibitem{SCN}
S.~Choi and S.~Nelson.
\newblock Mc-biquandles and mc-biquandle coloring quivers.
\newblock {\em arXiv:2311.00510}, 2023.

\bibitem{EN}
M.~Elhamdadi and S.~Nelson.
\newblock {\em Quandles---an introduction to the algebra of knots}, volume~74
  of {\em Student Mathematical Library}.
\newblock American Mathematical Society, Providence, RI, 2015.

\bibitem{FN}
P.~C. Falkenburg and S.~Nelson.
\newblock Biquandle bracket quivers.
\newblock {\em arXiv:2109.05365}, 2021.

\bibitem{IN}
K.~Istanbouli and S.~Nelson.
\newblock Quandle module quivers.
\newblock {\em J. Knot Theory Ramifications}, 29(12):2050084, 14, 2020.

\bibitem{J}
D.~Joyce.
\newblock A classifying invariant of knots, the knot quandle.
\newblock {\em J. Pure Appl. Algebra}, 23(1):37--65, 1982.

\bibitem{KNS}
J.~Kim, S.~Nelson, and M.~Seo.
\newblock Quandle coloring quivers of surface-links.
\newblock {\em J. Knot Theory Ramifications}, 30(1):Paper No. 2150002, 13,
  2021.

\end{thebibliography}
\bibliographystyle{abbrv}

\bigskip

\noindent
\textsc{Department of Mathematics and Statistics\\
Pomona College \\
333 N. College Way\\
Claremont, CA 91711
}

\medskip

\noindent
\textsc{Department of Mathematical Sciences \\
Claremont McKenna College \\
850 Columbia Ave. \\
Claremont, CA 91711}

\end{document}